\documentclass{amsart}
\usepackage[pdftex]{graphicx}
\usepackage{amsmath}
\usepackage{amssymb}
\usepackage{amsfonts}
\usepackage{color}
\usepackage[pdftex]{graphicx}
\usepackage{amsmath}
\usepackage{amssymb}
\usepackage{amsfonts}
\usepackage{color}
\newtheorem{theorem}{Theorem}[section]
\newtheorem{theor}{Theorem}

\newtheorem{lm}[theorem]{Lemma}

\newtheorem{pr}[theorem]{Proposition}

\newcommand{\ov}{\overrightarrow}
\begin{document}
\title{Routh's theorem for simplices}
\author{Franti\v sek Marko}
\address{Penn State Hazleton, 76 University Drive, Hazleton, PA 18202, USA}
\email{fxm13@psu.edu}
\author{Semyon Litvinov}
\address{Penn State Hazleton, 76 University Drive, Hazleton, PA 18202, USA}
\email{snl2@psu.edu}
\begin{abstract}
It is shown in \cite{lm} that, using only tools of elementary geometry, the classical Routh's theorem for triangles can be fully extended to tetrahedra.
In this article we first give another proof of Routh's theorem for tetrahedra where methods of elementary geometry are combined with the inclusion-exclusion principle.
Then we generalize this approach to $(n-1)-$dimensional simplices. 
A comparison with the formula obtained using vector analysis yields an interesting algebraic identity.
\end{abstract}
\keywords{Routh's theorem, inclusion-exclusion principle, tetrahedra, \\ $(n-1)-$dimensional simplices}
\subjclass[2010]{97G30}

\maketitle

\section*{Introduction}

The classical Routh's theorem, see rider (vii) on page 33 of \cite{g}\footnote{The authors would like to thank Mark B. Villarino of the University of Costa Rica for bringing this reference to their attention.} or page 82 of \cite{r}, states the following.

\begin{theor}
Let $ABC$ be an arbitrary triangle of area $1$, a point $K$ lie on the line segment $BC$, 
a point $L$ lie on the line segment $AC$ and a point $M$ lie on the line segment $AB$ 
such that $\frac{|AM|}{|MB|}=x$, $\frac{|BK|}{|KC|}=y$ and $\frac{|CL|}{|LA|}=z$. 
Denote by $P$ the point of intersection of lines $AK$ and $CM$, by $Q$ the point of intersection 
of lines $BL$ and $AK$, and by $R$ the point of intersection of lines $CM$ and $BL$ - see Figure \ref{F1}.
\begin{figure}[h]\centering
\includegraphics[height=2.5in]{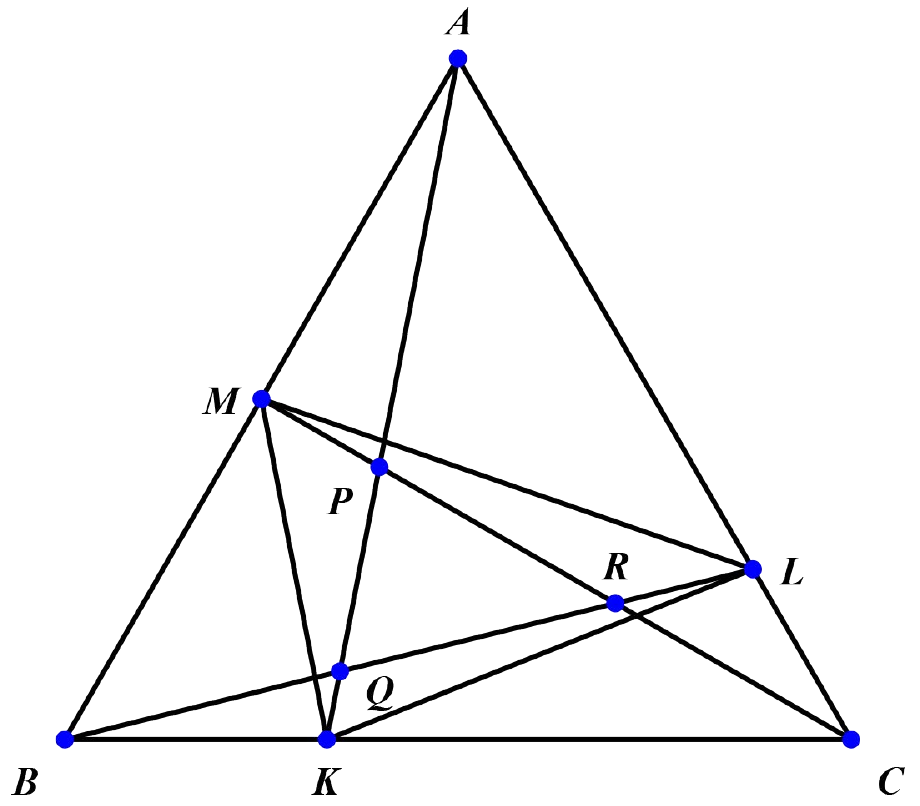}
\caption{Routh's Triangles}
\label{F1}
\end{figure}

Then the area of the triangle $KLM$ is 
$$
\frac{1+xyz}{(1+x)(1+y)(1+z)},
$$
and the area of the triangle 
$PQR$ is \[\frac{(1-xyz)^2}{(1+x+xy)(1+y+yz)(1+z+zx)}.\]
\end{theor}

Routh's theorem implies the theorem of Ceva:

\begin{theor}\label{Ceva}
The lines $AK$, $BL$ and $CM$ intersect at one point if and only if $xyz=1$.
\end{theor}

Routh's theorem is also closely related to the following theorem of Menelaus.

\begin{theor}\label{Menelaus}
Let $K$ be an arbitrary point on the line $BC$, $L$ on line $BC$ and $M$ on line $AB$. Denote 
$\frac{\ov{AM}}{\ov{MB}}=x$, $\frac{\ov{BK}}{\ov{KC}}=y$ and $\frac{\ov{CL}}{\ov{LA}}=z$.
Then the points $K,L,M$ are colinear if and only if $xyz=-1$.  
\end{theor}

We have conducted an extensive search of the literature on theorems of Routh, Ceva and Menelaus 
and their generalizations to higher dimensions within the context of Euclidean geometry (there are generalizations in other 
	geometries but we did not include them here) which resulted in the bibliography of the present paper. We believe that this list of articles is interesting from the historical perspective (although we cannot guarantee its completeness) and is valuable since it represents the wide range of generalizations of these classical theorems. 
We have been able to find only two papers, \cite{k} and \cite{yq}, where Routh's theorem is generalized to higher dimensions. 
Note that the statement in \cite{yq} is missing an absolute value
and the statement in \cite{k} needs to be reformulated to fit our notation. Unfortunately, both papers are not readily accessible to most of the readers since they are written in Chinese and Slovak languages, respectively. In \cite{lm}, we gave a geometric proof of Routh's theorem for tetrahedra. 
Keeping in mind that we would like to generalize this theorem to simplices, we need to adjust the notation as follows. 

\begin{theor}\label{tetra-new}
Let $A_1A_2A_3A_4$ be an arbitrary tetrahedron of volume $1$. Choose a point $P_1$ on the edge $A_1A_2$, a point $P_2$ on the edge $A_2A_3$, 
a point $P_3$ on the edge $A_3A_4$, and a point $P_4$ on the edge $A_4A_1$ such that $\frac{|P_1A_1|}{|P_1A_2|}=x_1$, 
$\frac{|P_2A_2|}{|P_2A_3|}=x_2$, $\frac{|P_3A_3|}{|P_3A_4|}=x_3$, and $\frac{|P_4A_4|}{|P_4A_1|}=x_4$. 
Then
\begin{equation}\label{e-2}
V_{P_1P_2P_3P_4}=\frac{|1-x_1x_2x_3x_4|}{(1+x_1)(1+x_2)(1+x_3)(1+x_4)}.
\end{equation}

The four planes given by the points $A_1, A_2, P_3$, points $A_2,A_3,P_4$, 
points $A_3,A_4,P_1$, and points $A_4,A_1,P_2$ enclose the tetrahedron $R_1R_2R_3R_4$ (see Figure \ref{F2}) of the volume

\begin{equation}\label{e-1}
\begin{split}
V_{R_1R_2R_3R_4}=\frac{|1-x_1x_2x_3x_4|^3}{(1+x_1+x_1x_2+x_1x_2x_3)(1+x_2+x_2x_3+x_2x_3x_4)} \times \\
\frac 1{(1+x_3+x_3x_4+x_3x_4x_1)(1+x_4+x_4x_1+x_4x_1x_2)}.
\end{split}
\end{equation}
\end{theor}

\begin{figure}[h]\centering
\includegraphics[height=3.5in]{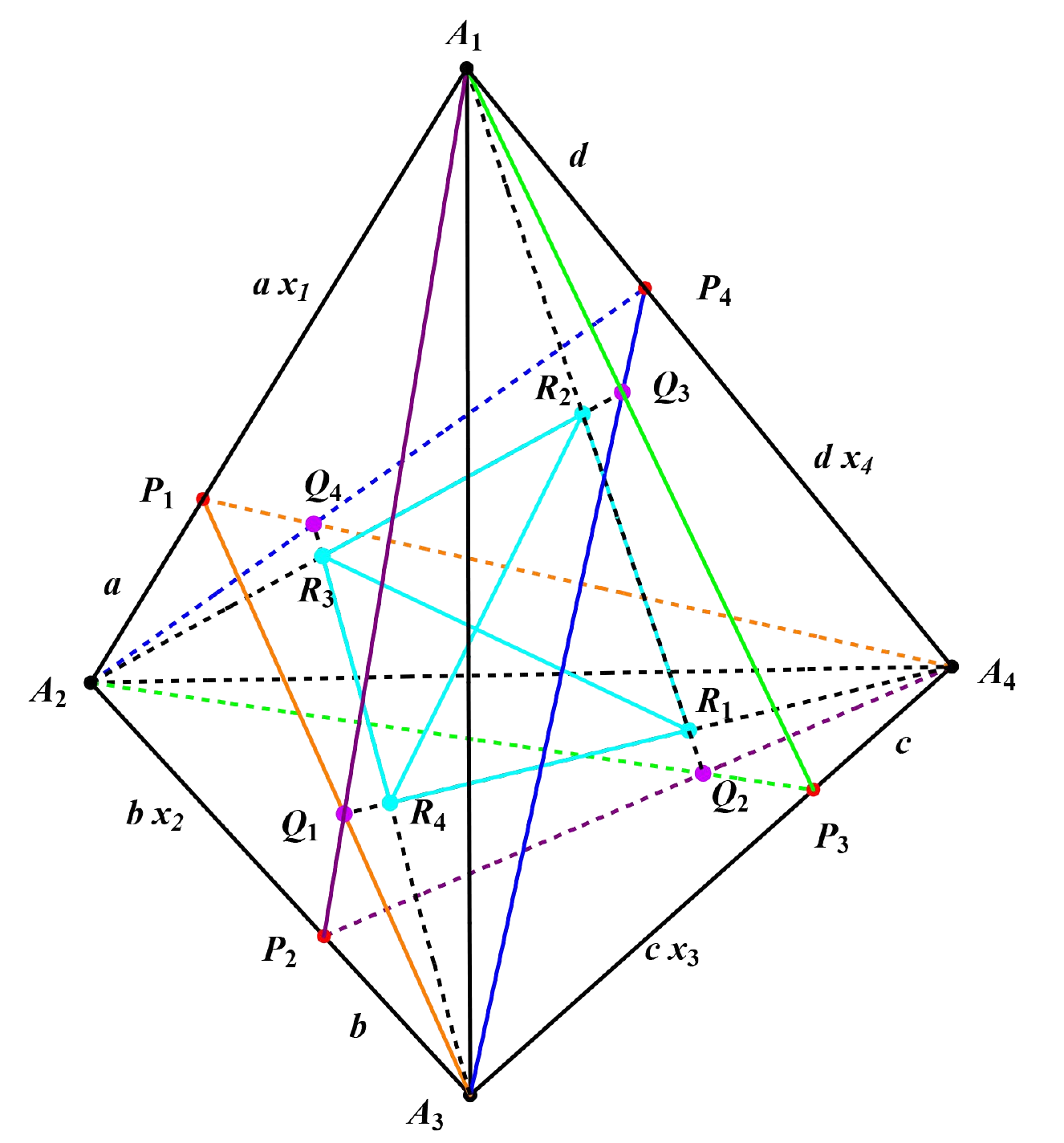}
\caption{Notation, $x_1x_2x_3x_4>1$}
\label{F2}
\end{figure}

The formulas in the above theorem correspond to the cycle $(A_1A_2A_3A_4)$. Opposite to \cite{lm}, we will assume that $x_1x_2x_3x_4>1$.
If $x_1x_2x_3x_4<1$, then we can change the orientation of the cycle $(A_1A_2A_3A_4)$. As a consequence, the product 
$x_1x_2x_3x_4$ will change to $\frac{1}{x_1x_2x_3x_4}>1$, and a simple evaluation leads to the same reesult. 

Formula (\ref{e-1}) will be proved using geometric considerations together with the principle of inclusion-exclusion. 

In Section 2 of the article we will extend our considerations to the cycle $(A_1\ldots A_n)$ corresponding to a general 
$(n-1)-$dimensional simplex $A_1 \ldots A_n$. Comparison of the result with the formula given in \cite{yq} yields a remarkable algebraic identity (\ref{e2}).


\section{Routh's theorem for tetrahedra: proof of (\ref{e-1})}
Let us assume that $x_1x_2x_3x_4 >1$.
To the cutting plane $\sigma_1$ given by points $A_3, A_4, P_1$ we assign the half-space $S_1$ containing $A_1$, 
to the cutting plane $\sigma_2$ given by points $A_1, A_4, P_2$ we assign the half-space $S_2$ containing $A_2$, 
to the cutting plane $\sigma_3$ given by points $A_1, A_2, P_3$ we assign the half-space $S_3$ containing $A_3$, and 
to the cutting plane $\sigma_4$ given by points $A_2, A_3, P_4$ we assign the half-space $S_4$ containing $A_4$.
For $i=1,2,3,4,$ denote by $T_i$ the tetrahedron that is the intersection of $S_i$ with the tetrahedron $A_1A_2A_3A_4$, and by $V_i$ the volume of the tetrahedron $T_i$.

If $x_1x_2x_3x_4>1$, then the intersection $T_1\cap T_2\cap T_3\cap T_4=S_1\cap S_2\cap S_3 \cap S_4$ is the tetrahedron $R_1R_2R_3R_4$. (If $x_1x_2x_3x_4=1$, then this intersection is a single point and if $x_1x_2x_3x_4<1$, this intersection if empty).

In what follows we denote the volume of a tetrahedron $T=ABCD$ by $V_{ABCD}$ or $V_T$ and analogously for other tetrahedra. 
For the convenience of the reader we shall now restate Lemmas 6, 7, and 8 of \cite{lm}. 

\begin{lm}\label{lm0}
In the notation of Figure \ref{fig5},
$$
V_{AMCD}=V_{ABCD}\frac{|AM|}{|AB|}.
$$
\end{lm}

\begin{lm}\label{lm15} In the notation of Figure \ref{fig5}, 
\begin{figure}[ht]\centering
\includegraphics[height=2.5in]{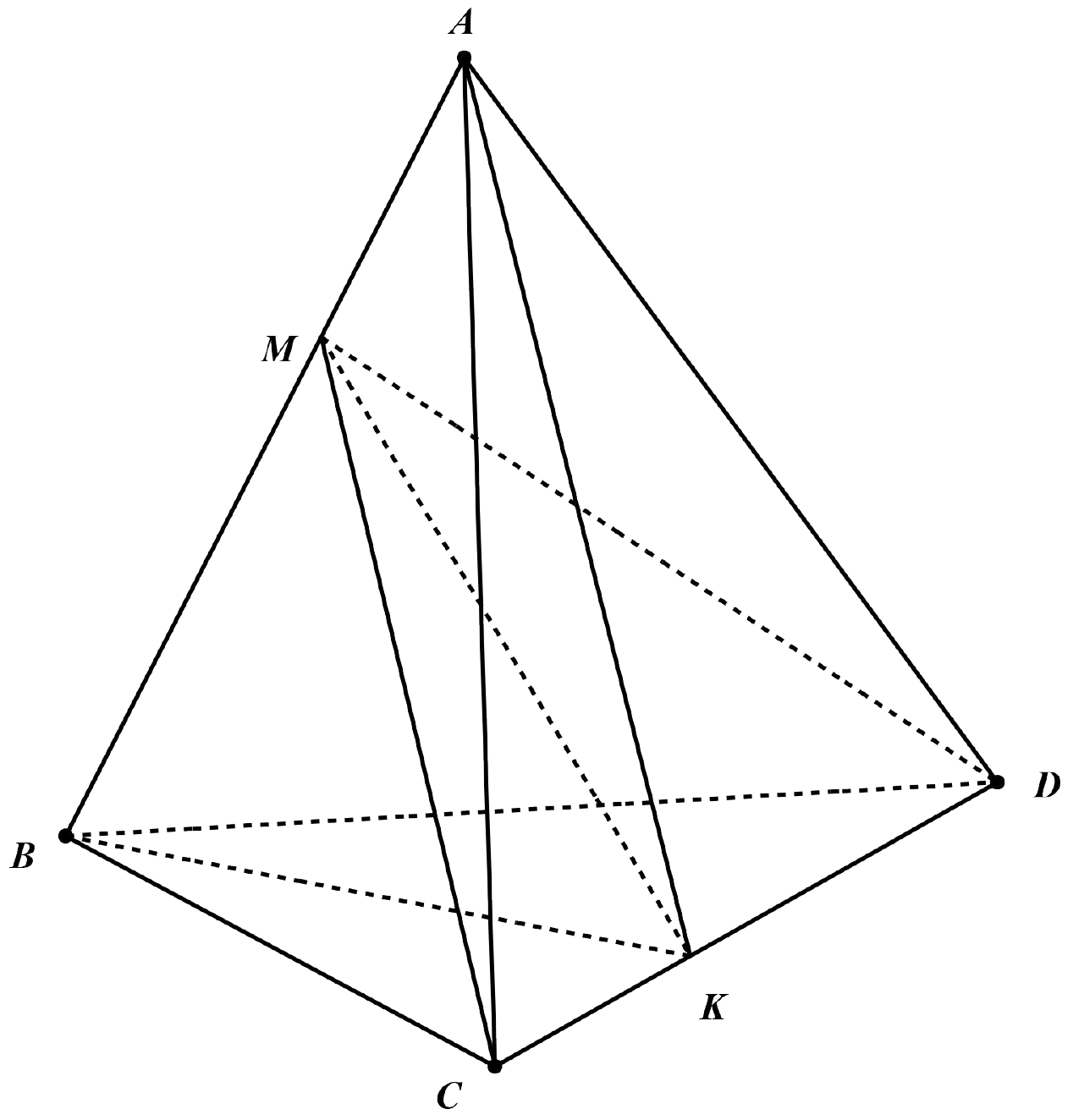}
\caption{Tetrahedra AMCD, ACKM, ADKM}
\label{fig5}
\end{figure}

\[V_{AKCM}=V_{ABCD}\frac {|AM|}{|AB|}\frac{|CK|}{|CD|} \text{ \ and \ } 
V_{AMKD}=V_{ABCD}\frac{|AM|}{|AB|}\frac{|DK|}{|DC|}. \]
\end{lm}

\begin{lm}\label{lm35}
Consider the triangle $ABC$ in Figure \ref{F3}. 
\begin{figure}[ht]\centering
\includegraphics[height=2.5in]{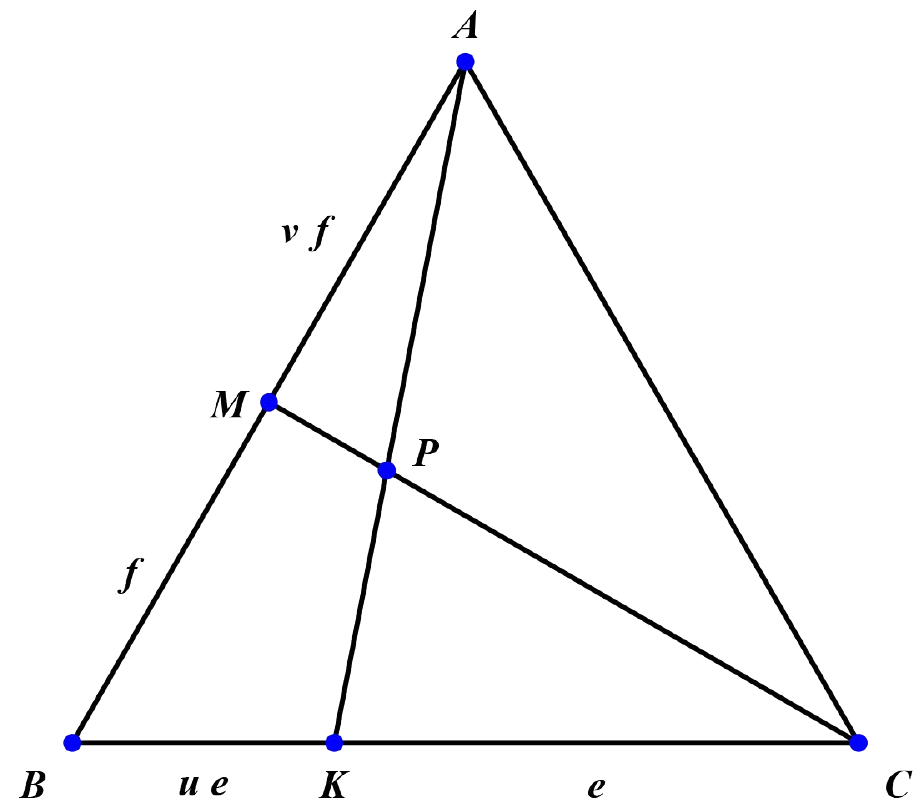}
\caption{Ratios}
\label{F3}
\end{figure}
If $\frac{|AM|}{|MB|}=v$ and $\frac{|BK|}{|KC|}=u$, then \[\frac{|AP|}{|PK|}=v(1+u).\]
\end{lm}

\vskip 5pt
Using Lemma \ref{lm0}, we obtain
$$
V_{T_1}=\frac{x_1}{1+x_1}, \ \ V_{T_2}=\frac{x_2}{1+x_2}, \ \ V_{T_3}=\frac{x_3}{1+x_3}, \ \ V_{T_4}=\frac{x_4}{1+x_4}.
$$
It follows from Lemma \ref{lm15} that 
$$
V_{T_1\cap T_3}=\frac{x_1}{1+x_1}\frac{x_3}{1+x_3} \text { \  and \ }
V_{T_2\cap T_4}=\frac{x_2}{1+x_2}\frac{x_4}{1+x_4}.
$$

\begin{lm}\label{lm1}
Consider the triangle in Figure \ref{F3}.
Then 
$$
\frac{|MP|}{|PC|}=\frac{vu}{1+v} \text { \ and \ } \frac{|MP|}{|MC|}=\frac{vu}{1+v+vu}.
$$
\end{lm}
\begin{proof}
Lemma \ref{lm35} applied to the triangle $CBA$ yields 
$$
\frac{|CP|}{|PM|}=\frac1u \left (1+\frac1v \right )=\frac{1+v}{uv},
$$
and the desired ratios follow.
\end{proof}
\begin{figure}[ht]\centering
\includegraphics[height=3in]{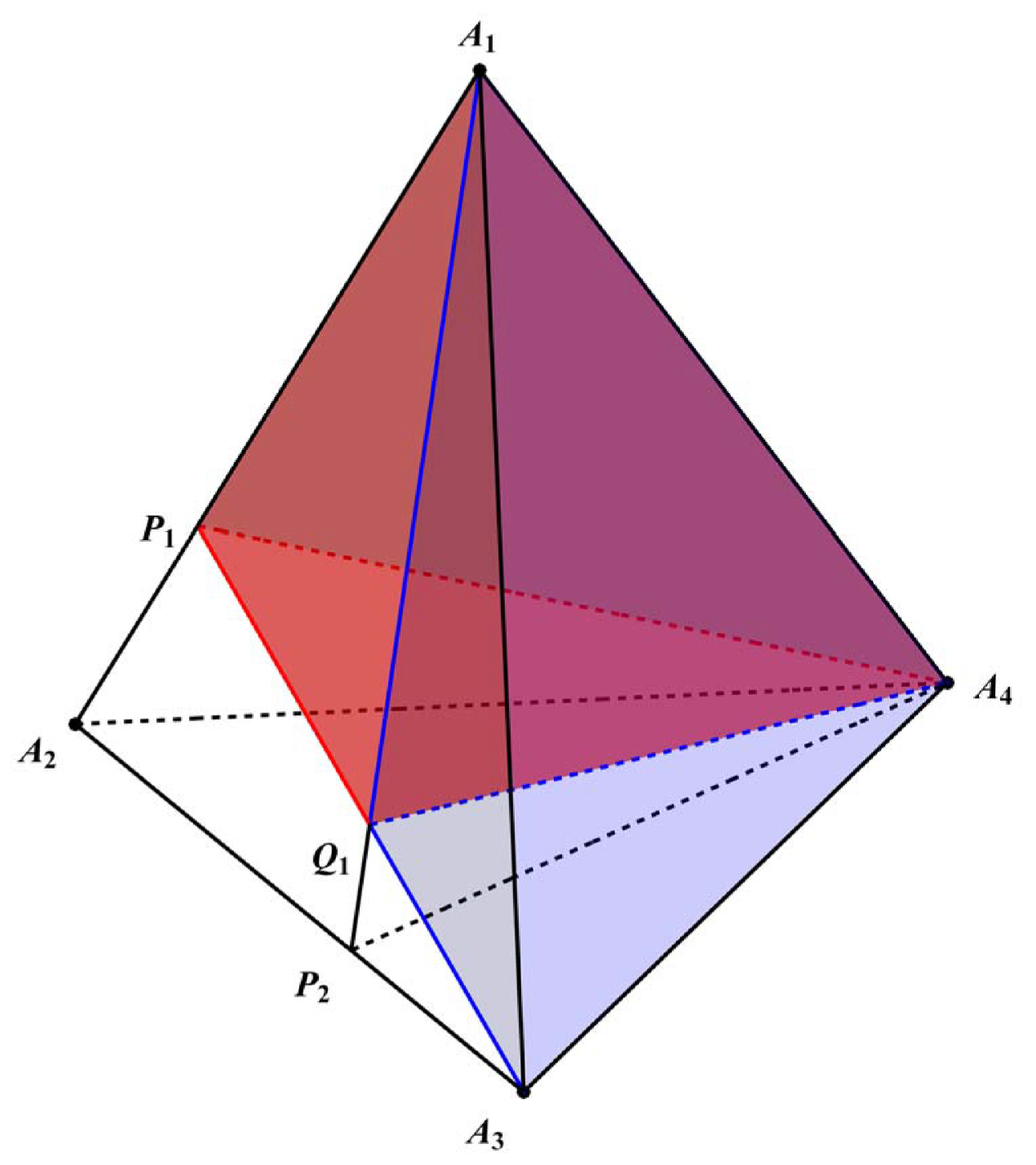}
\caption{Middle tetrahedron}
\label{F5}
\end{figure}
\begin{lm}\label{lm2}
The volume of $T_1\cap T_2$ is \[\frac{x_1^2x_2}{(1+x_1)(1+x_1+x_1x_2)}.\]
\end{lm}
\begin{proof}
It can be observed from Figure \ref{F5} 
that $T_1\cap T_2$ is the tetrahedron $A_1P_1Q_1A_4$.
Using Lemma \ref{lm0}, we obtain
$$
V_{A_1P_1A_3A_4}=\frac{x_1}{1+x_1} \text { \  and \ }  V_{A_1P_1Q_1A_4}=V_{A_1P_1A_3A_4}\frac{|P_1Q_1|}{|Q_1A_3|}.
$$
Since, by Lemma \ref{lm1}, \[\frac{|P_1Q_1|}{|Q_1A_3|}=\frac{x_1x_2}{1+x_1+x_1x_2},\] the formula follows.
\end{proof}

It follows as in the proof of Lemma \ref{lm2} that 
$$
V_{T_2\cap T_3}=\frac{x_2^2x_3}{(1+x_2)(1+x_2+x_2x_3)},
$$
$$
V_{T_3\cap T_4}=\frac{x_3^2x_4}{(1+x_3)(1+x_3+x_3x_4)},
$$
$$
V_{T_4\cap T_1}=\frac{x_4^2x_1}{(1+x_4)(1+x_4+x_4x_1)}.
$$

\begin{lm}\label{lm3}
The volume of $T_1\cap T_2\cap T_3$ is \[\frac{x_1^3x_2^2x_3}{(1+x_1)(1+x_1+x_1x_2)(1+x_1+x_1x_2+x_1x_2x_3)}.\]
\end{lm}
\begin{proof}
In the notation of Figure \ref{F2}, the intersection $T_1\cap T_2 \cap T_3$ is the tetrahedron $A_1 P_1 Q_1 R_1$. Looking at Figure \ref{F6} 
\begin{figure}[h]\centering
\includegraphics[height=3in]{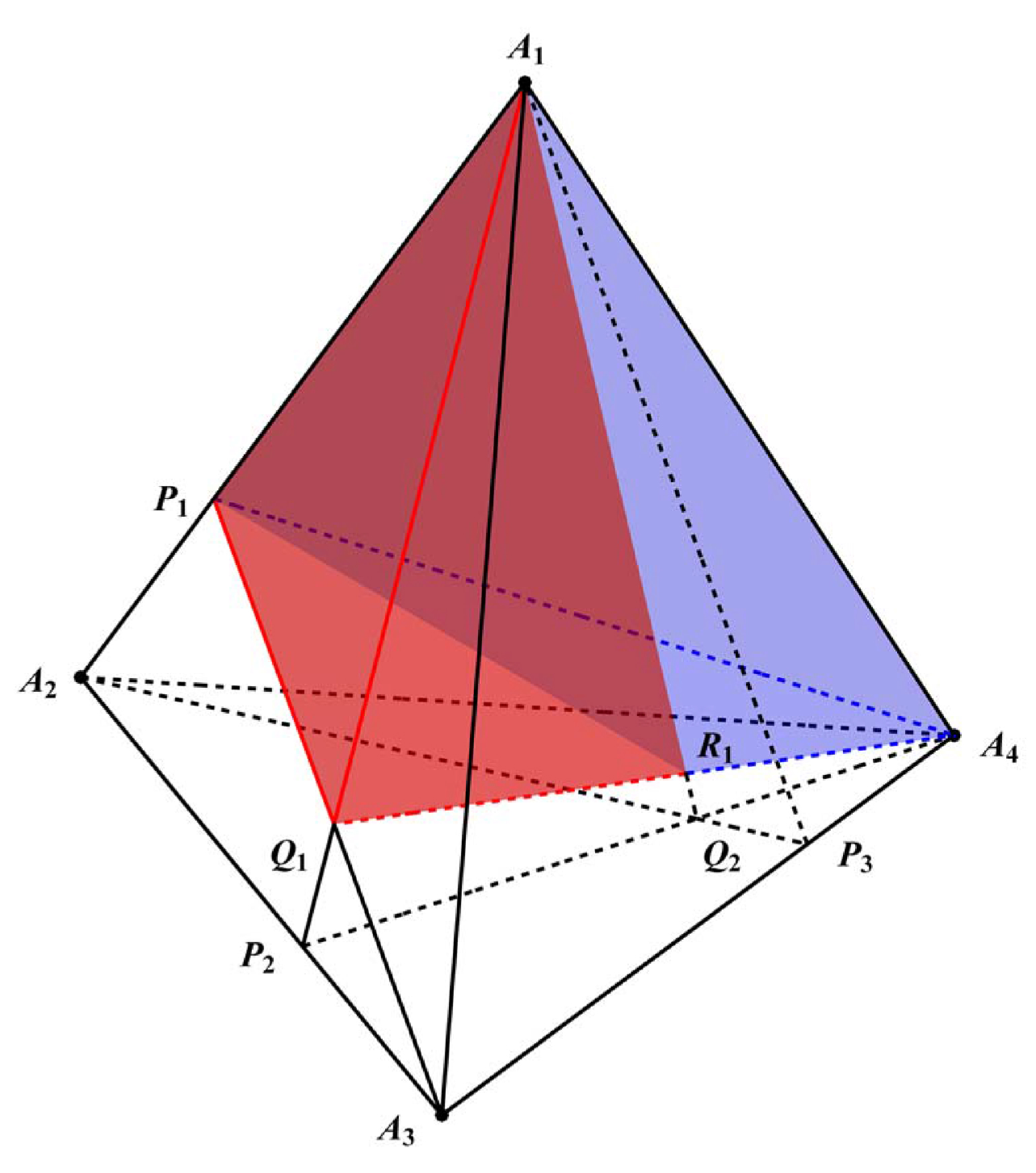}
\caption{Small tetrahedron}
\label{F6}
\end{figure}
and using Lemma \ref{lm0}, we determine that 
\[V_{A_1 P_1 Q_1 R_1}= V_{A_1 P_1 Q_1 A_4} \frac{|Q_1 R_1|}{|Q_1 A_4|},\] where 
\[V_{A_1 P_1 Q_1 A_4}=V_{T_1 \cap T_2} = \frac{x_1^2x_2}{(1+x_1)(1+x_1+x_1x_2)}\] by Lemma \ref{lm2}. 
To find the remaining ratio
$\frac{|Q_1 R_1|}{|Q_1 A_4|}$, consider the triangle $A_1 P_2 A_4$ as depicted in Figure \ref{F7}.
\begin{figure}[h]\centering
\includegraphics[height=2.5in]{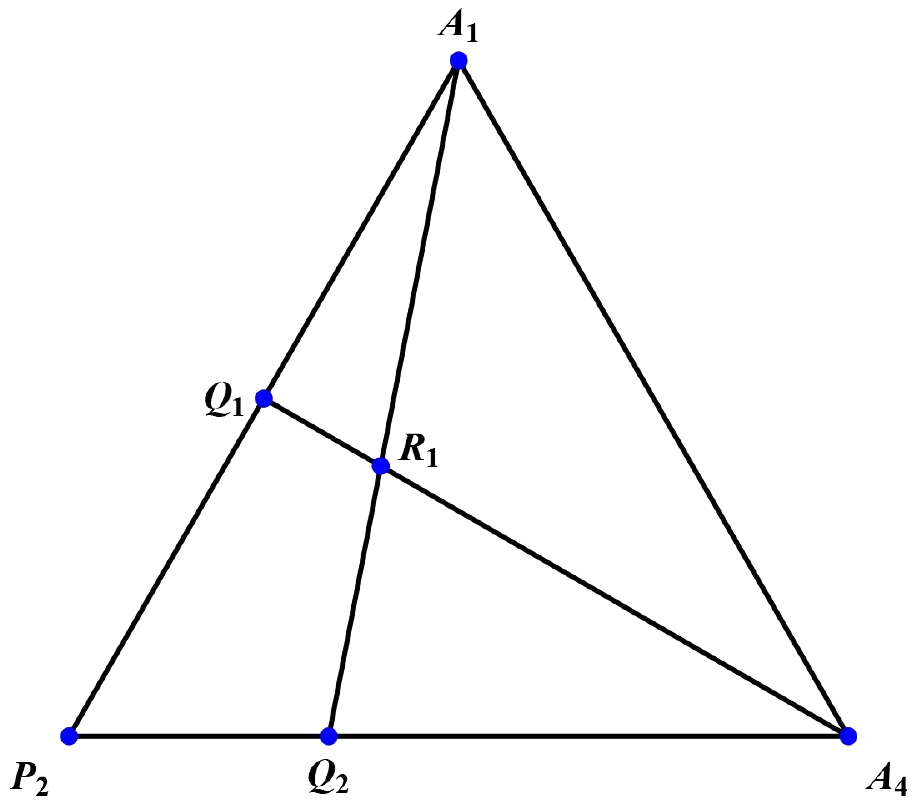}
\caption{Ratio $\frac{|Q_1 R_1|}{|Q_1 A_4|}$}
\label{F7}
\end{figure}
Here we have $v=x_1(1+x_2)$ by Lemma \ref{lm35} applied to the triangle $A_1 A_2 A_3$ and $u=\frac{x_2x_3}{1+x_2}$ 
by Lemma \ref{lm1} applied to the triangle
$A_2 A_3 A_4$.  Therefore Lemma \ref{lm1} applied to the triangle $A_1 P_2 A_4$ yields 
\[\frac{|Q_1 R_1|}{|Q_1 A_4|}=\frac{vu}{1+v+vu}=\frac{x_1x_2x_3}{1+x_1+x_1x_2+ x_1x_2x_3},\]
and the result follows.
\end{proof}

It follows as in the proof of Lemma \ref{lm3} that
$$
V_{T_2\cap T_3 \cap T_4}=\frac{x_2^3x_3^2x_4}{(1+x_2)(1+x_2+x_2x_3)(1+x_2+x_2x_3+x_2x_3x_4)},
$$
$$
V_{T_3\cap T_4 \cap T_1}=\frac{x_3^3x_4^2x_1}{(1+x_3)(1+x_3+x_3x_4)(1+x_3+x_3x_4+x_3x_4x_1)},
$$
$$
V_{T_4\cap T_1 \cap T_2}=\frac{x_4^3x_1^2x_2}{(1+x_4)(1+x_4+x_4x_1)(1+x_4+x_4x_1+x_4x_1x_2)}.
$$

\bigskip

{\bf Proof of (\ref{e-1}) in Theorem \ref{tetra-new}:}
Assume first that $x_1 x_2 x_3 x_4 >1$.
Using the principle of inclusion-exclusion, we obtain
\[\begin{aligned}
V_{R_1 R_2 R_3 R_4}= &V_{A_1 A_2 A_3 A_4} - V_{T_1}-V_{T_2}-V_{T_3}-V_{T_4}\\
&+V_{T_1\cap T_3}+V_{T_2\cap T_4} +V_{T_1 \cap T_2} + V_{T_2 \cap T_3}+ V_{T_3 \cap T_4} +V_{T_4 \cap T_1} \\
&-V_{T_1\cap T_2 \cap T_3} -V_{T_2 \cap T_3 \cap T_4}-V_{T_3 \cap T_4 \cap T_1} - V_{T_4 \cap T_1 \cap T_2}.
\end{aligned}\]

Formula (\ref{e-1}) now follows from the previous formulas for the above volumes together with the following identity (\ref{ie}).

\begin{equation}\label{ie}
\begin{aligned}
&1-\frac{x_1}{1+x_1}-\frac{x_2}{1+x_2}-\frac{x_3}{1+x_3}-\frac{x_4}{1+x_4}
+\frac{x_1^2x_2}{(1+x_1)(1+x_1+x_1x_2)}\\
&+\frac{x_2^2x_3}{(1+x_2)(1+x_2+x_2x_3)}
+\frac{x_3^2x_4}{(1+x_3)(1+x_3+x_3x_4)}\\
&+\frac{x_4^2x_1}{(1+x_4)(1+x_4+x_4x_1)}
+\frac{x_1x_3}{(1+x_1)(1+x_3)}+\frac{x_2x_4}{(1+x_2)(1+x_4)}\\
&-\frac{x_1^3x_2^2x_3}{(1+x_1)(1+x_1+x_1x_2)(1+x_1+x_1x_2+x_1x_2x_3)}\\
&-\frac{x_2^3x_3^2x_4}{(1+x_2)(1+x_2+x_2x_3)(1+x_2+x_2x_3+x_2x_3x_4)}\\
&-\frac{x_3^3x_4^2x_1}{(1+x_3)(1+x_3+x_3x_4)(1+x_3+x_3x_4+x_3x_4x_1)}\\
&-\frac{x_4^3x_1^2x_2}{(1+x_4)(1+x_4+x_4x_1)(1+x_4+x_4x_1+x_4x_1x_2)}\\
&=\frac{(x_1x_2x_3x_4-1)^3}{(1+x_1+x_1x_2+x_1x_2x_3)(1+x_2+x_2x_3+x_2x_3x_4)}\times\\
&\quad \,\, \frac{1}{(1+x_3+x_3x_4+x_3x_4x_1)(1+x_4+x_4x_1+x_4x_1x_2)}.
\end{aligned}
\end{equation}

\vskip 5pt
\noindent Identity (\ref{ie}) can be verified either manually or by using a software like Mathematica or Maple.

The case $x_1 x_2 x_3 x_4 <1$ can be treated similarly to \cite{lm} by reversing the orientation of the cycle $(A_1 A_2 A_3 A_4)$ to $(A_1 A_4 A_3 A_2)$ and 
using the substitution $x_1\mapsto \frac{1}{x_4}$, $x_2\mapsto \frac{1}{x_1}$, $x_3\mapsto \frac{1}{x_2}$, $x_4\mapsto \frac{1}{x_3}$
that reduces it to the case $x_1 x_2 x_3 x_4 >1$.
\qed

\section{Routh's theorem for simplices and related algebraic identities}

Reviewing the formulas for the volumes of various tetrahedra appearing in the application of the inclusion-exclusion principle in Section 1, it is possible to observe the pattern that holds in the general case of an $(n-1)$-dimensional simplex \\ $S=A^0_1 \ldots A^0_n$. (In this section we assume that $n\geq 4$.)

We will work with the cycle $(A^0_1 \ldots A^0_n)$ and for simplicity of notation we will consider all indices modulo $n$, that is, we identify the index $n+1$ with $1$, and so on.
For each $i=1, \ldots, n$ choose a point $A^1_i$ on the edge $A^0_i A^0_{i+1}$ of $S$ and denote $\frac{|A^0_i A^1_i|}{|A^1_i A^0_{i+1}|}=x_i$.
Let $H_i$ be the half-space given by the hyperplane $\sigma_i$ containing points $A^0_1 \ldots A^0_{i-1} A^1_i A^0_{i+2} \ldots A^0_n$ in the direction of the point $A^0_i$, and 
$T_i$ the intersection of $H_i$ with the original simplex $S$.  We will assume that $\prod_{i=1}^n x_i>1$.  In this case, the intersection of all half-spaces 
$H_i$ and $S$ is the $(n-1)$-dimensional simplex $\bigcap_{i=1}^n T_i$. 

We will obtain a generalization of Routh's theorem by determining a formula for the volume of the simplex $\bigcap_{i=1}^n T_i$ in terms of $x_i$'s. 

An additional notation is in order. For each $i=1, \ldots, n$ and $j=2, \ldots, n-1$, let $A^j_i$ be the point of the intersection of 
the lines $A^{j-1}_i A^0_{i+j}$ and $A^{j-1}_{i+1}A^0_i$. 

Our argument will rely heavily on the triangles described in the following Lemma \ref{cuts} and the ratios calculated in Lemma \ref{lgen}.

\begin{lm}\label{cuts}
For every $i=1, \ldots, n$ and $j=2, \ldots, n-1$ consider the triangle $A^0_i A^{j-2}_{i+1} A^0_{i+j}$ with the point $A^{j-1}_i$ on the edge $A^0_i A^{j-2}_{i+1}$, 
the point $A^{j-1}_{i+1}$ on the edge $A^{j-2}_{i+1} A^0_{i+j}$, and the point $A^j_i$ as depicted in Figure \ref{F8}.

\begin{figure}[h]\centering
\includegraphics[height=3in]{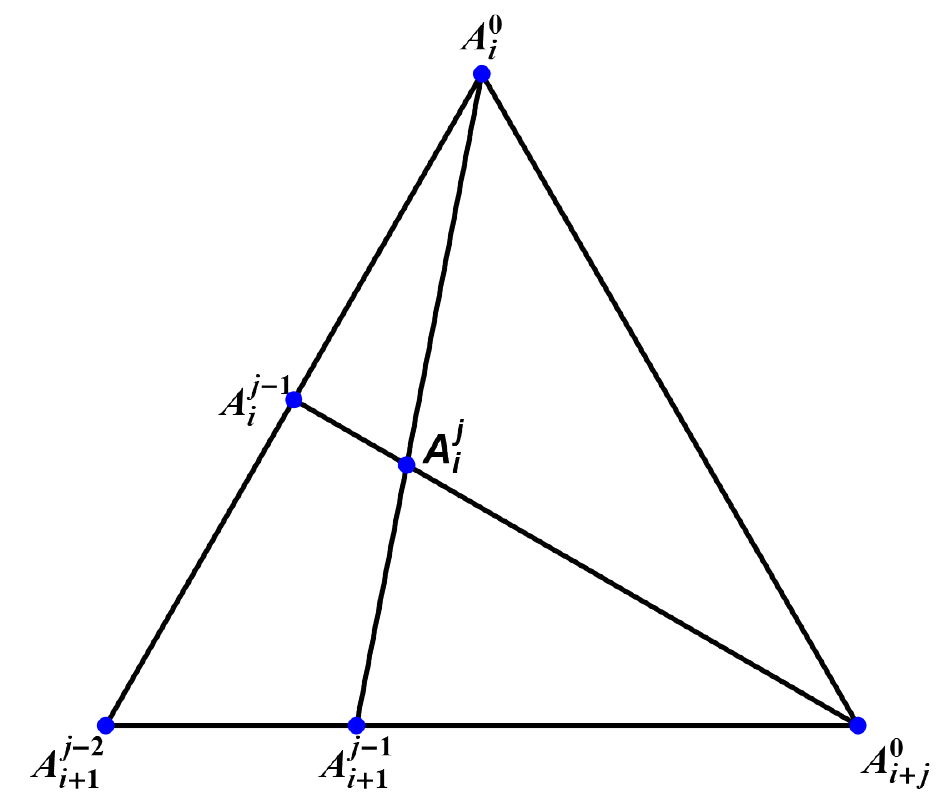}
\caption{General position}
\label{F8}
\end{figure}

Then the line $A^0_iA^{j-1}_{i+1}$ and the point $A^j_i$ belong to the hyperplane $\sigma_{i+j-1}$.
\end{lm}
\begin{proof}
Note that all points in Figure \ref{F8} belong to the same plane. For $j=2$, it follows from the choice of the points $A^1_i$ and $A^1_{i+1}$. For $j>2$, 
it follows by induction on $j$ since $A^{j-1}_i$ lies on the line given by $A^0_{i}$ and $A^{j-2}_{i+1}$, and $A^{j-1}_{i+1}$ lies on the line given by $A^0_{i+j}$ and $A^{j-2}_{i+1}$.

The statement of the lemma is true for $j=2$ since both $A^0_i$ and $A^1_{i+1}$ belong to $\sigma_{i+1}$. For $j>2$, it follows by induction on $j$ since the point $A^0_i$ 
belongs to $\sigma_{i+j-1}$ and the point 
$A^{j-1}_{i+1}$ belongs to $\sigma_{i+j-1}$ by induction applied to the triangle $A^0_{i+1}A^{j-3}_{i+2}A^0_{i+j}$.
\end{proof}

In relation to the triangle in Figure \ref{F8}, define 
\[v_{i,j}=\frac{|A^0_i A^{j-1}_i|}{|A^0_i A^{j-2}_{i+1}|}, \ \ u_{i,j}=\frac{|A^{j-2}_{i+1}A^{j-1}_{i+1}|}{|A^{j-2}_{i+1} A^0_{i+j}|}, \ \ 
t_{i,j}=\frac{|A^j_iA^{j-1}_i|}{|A^0_{i+j}A^{j-1}_i|}.\]

\begin{lm}\label{lgen}
Given $i=1, \ldots n$ and $j=2, \ldots, n-1$, we have 
\[v_{i,j}=x_i+x_i x_{i+1} + \ldots + x_ix_{i+1} \ldots x_{i+j-1},\]
\[u_{i,j}=\frac{x_{i+1}\ldots x_{i+j}}{1+x_{i+1}+x_{i+1}x_{i+2}+ \ldots + x_{i+1}\ldots x_{i+j-1}},\]
\[t_{i,j}=\frac{x_{i}\ldots x_{i+j-1}}{1+x_{i}+x_{i}x_{i+1}+ \ldots + x_{i}\ldots x_{i+j-1}}.\]
\end{lm}
\begin{proof}
We proceed by induction on $j$. 
For $j=2$ the statements $v_{i,2}=x_ix_{i+1}$, $u_{i,2}=\frac{x_{i+1}x_{i+2}}{1+x_{i+1}}$, and $t_{i,2}=\frac{x_i x_{i+1}}{1+x_i+x_i x_{i+1}}$
follow in the same way as in the proof of Lemma \ref{lm2}.
The formulae $v_{i,3}=x_ix_{i+1}x_{i+2}$, $u_{i,3}=\frac{x_{i+1}x_{i+2}x_{i+3}}{1+x_{i+1}+x_{i+1}x_{i+2}}$, and \\ $t_{i,2}=\frac{x_i x_{i+1}x_{i+2}}{1+x_i+x_i x_{i+1}+x_ix_{i+1}x_{i+2}}$
for $j=3$ follow in the same way as in the proof of Lemma \ref{lm3}.

For the inductive step, apply Lemmas \ref{lm35} and \ref{lm1} to the triangle on Figure \ref{F8} with $v=v_{i,j-1}$ and $u=u_{i,j-1}$. By Lemma \ref{lm35}, we infer that
\[v_{i,j}=v_{i,j-1}(1+u_{i,j-1})=x_i+x_i x_{i+1} + \ldots + x_ix_{i+1}\ldots x_{i+j-1}.\] By Lemma \ref{lm1}, we obtain 
\[u_{i,j}=\frac{v_{i,j-1}u_{i,j-1}}{1+v_{i,j-1}+v_{i,j-1}u_{i,j-1}}=\frac{x_{i+1}\ldots x_{i+j}}{1+x_{i+1}+x_{i+1}x_{i+2}+ \ldots + x_{i+1}\ldots x_{i+j-1}}.\]

Finally, the formula for $t_{i,j}$ follows from the formulas for $v_{i,j}$ and $u_{i,j}$ using Lemma \ref{lm1}.
\end{proof}

We will determine the volume of the simplex $\bigcap_{i=1}^n T_i$ (with the vertices $A^{n-1}_1, \ldots, $ $A^{n-1}_n)$ using the inclusion-exclusion principle.
For this we shall compute the volumes of all simplices $\bigcap_{i\in I} T_i$, where $I\subsetneq \{1, \ldots, n\}$. An important property of such simplices
$\bigcap_{i\in I} T_i$ is that they contain the original vertices $A^0_j$, where $j\notin I$. 

From now on, we will assume that $I\subsetneq \{1, \ldots, n\}$.
We now proceed to determine the vertices of the simplices $\bigcap_{i\in I} T_i$ and compute their volumes. 
When calculating the volume of $\bigcap_{i \in I} T_i$, the crucial role is played by the distribution of elements $i\in I$ along the cycle $C=(1 \ldots n)$. 
Assume that the set $I$ consists of blocks of consecutive elements along the cycle $C$ and keep in mind that a block containing $n$ can start before 
$n$ and continue through $n$ to $1$ and further. Denote by $\mathcal{B}(I)$ the set of all blocks of $I$ along the cycle $C$.
To each block of $I$, say $B=\{k, k+1, \ldots, k+l\}$, we assign the expression 
\[V(B)=\prod_{j=k}^{k+l} \frac{\prod_{a=k}^j x_a}{1+ \sum_{b=k}^{j} \prod_{a=k}^b x_a}.\]

\noindent
For example, 

\[\begin{aligned}V(\{2,3,4,5\}) = &\frac{x_2}{1+x_2}\frac{x_2x_3}{1+x_2+x_2x_3}\frac{x_2x_3x_4}{1+x_2+x_2x_3+x_2x_3x_4}\times \\
&\frac{x_2x_3x_4x_5}{1+x_2+x_2x_3+x_2x_3x_4+x_2x_3x_4x_5}.
\end{aligned}\]

\vskip 8pt
\begin{pr}\label{seq}
Let $B=\{k, k+1, \ldots, k+l\}$ be a subset of $I$ (hence $l<n-1$). Then 
the vertices of the simplex $\bigcap_{i \in B} T_i$ are $A^0_k, A^1_k, \ldots, A^{l+1}_k, A^0_{k+l+2}, \ldots, A^0_{n+k-1}$ and 
\[V_{\bigcap_{i \in B} T_i}=V(B)=\prod_{j=k}^{k+l} \frac{\prod_{a=k}^j x_a}{1+ \sum_{b=k}^{j} \prod_{a=k}^b x_a}\]
\end{pr}
\begin{proof}
We proceed by induction on $l$. If $l=0$, the statement $V_{T_k}=\frac{x_k}{1+x_k}$ follows from Lemma \ref{lm0}.

\noindent
Assume that $l>0$, the vertices of $\bigcap_{i=k}^{k+l-1} T_i$ are $A^0_k, A^1_k, \ldots, A^l_k, A^0_{k+l+1}, \ldots, A^0_{n+k-1}$,
and \[V_{\bigcap_{i=k}^{k+l-1} T_i}=\prod_{j=k}^{k+l-1} \frac{\prod_{a=k}^j x_a}{1+ \sum_{b=k}^{j} \prod_{a=k}^b x_a}.\]
Since $T_{k+l}$ has the vertices $A^0_k, \ldots, A^0_{k+l}, A^1_{k+l}, A^0_{k+l+2}, \ldots, A^0_{n+k-1}$, 
and all vertices $A^0_k, A^1_k, \ldots, A^l_k$ are included in the simplex given by vertices $A^0_k, \ldots, A^0_{k+l}$,
when we cut the simplex $\bigcap_{i=k}^{k+l-1} T_i$ by $H_{k+l}$, all of its edges remain the same except the edge 
$A^l_k A^0_{k+l+1}$. The edge $A^l_k A^0_{k+l+1}$ is replaced by the edge $A^l_k A^{l+1}_k$ as can be seen from Lemma \ref{cuts} and Figure \ref{F8} because $A^{l+1}_k$ belongs to $\sigma_{k+l}$.
Therefore the vertices of $\bigcap_{i \in B} T_i$ are $A^0_k, A^1_k, \ldots, A^{l+1}_k, A^0_{k+l+2}, \ldots, A^0_{n+k-1}$.

Since the vertices $A^l_k$, $A^{l+1}_k$, $A^0_{k+l+1}$ lie on the same line, using 
Lemma \ref{lm0}, we derive that 
\[V_{\bigcap_{i=k}^{k+l} T_i}=V_{\bigcap_{i=k}^{k+l-1} T_i} \frac{|A^{l+1}_k A^l_k|}{|A^0_{k+l+1}A^l_k|}.\] 
Since $\frac{|A^{l+1}_k A^l_k|}{|A^0_{k+l+1}A^l_k|}=t_{k,l+1}$, Lemma \ref{lgen} concludes the inductive step.
\end{proof}

\vskip 5pt
To find $V_{\bigcap_{i\in I} T_i}$, we need to understand the role of blocks. 
Write a proper subset $I$ of $\{1, \ldots, n\}$ as a disjoint union of its blocks
\[I=\bigcup\limits_{B\in \mathcal{B}(I)} B = B_1\cup \ldots \cup B_s=\{k_1, \ldots, k_1+l_1\}\cup \ldots \cup \{k_s, \ldots, k_s+l_s\},\]
$k_1<\ldots<k_s$.

We will show that the vertices of the simplex $\bigcap_{i\in I} T_i$ are 
\begin{equation}\label{sequence}
\begin{aligned}&A^0_{k_1}, A^1_{k_1}, \ldots, A^{l_1+1}_{k_1}, A^0_{k_1+l_1+2}, \ldots, A^0_{k_2}, A^1_{k_2}, \ldots, A^{l_2+1}_{k_2}, A^0_{k_2+l_2+2}, \ldots \\
&A^0_{k_s}, A^1_{k_s}, \ldots, A^{l_s+1}_{k_s}, A^0_{k_s+l_s+2}, \ldots, A^0_{k_1+n-1}.
\end{aligned}
\end{equation} 

For $k_1\leq k<k_1+n$, denote $S^k_I=\bigcap_{i\in I|k_1\leq i\leq k} T_i$ and list these $(n-1)-$dimensional simplices in the order 
$S^{k_1}_I, S^{k_1+1}_I, \ldots, S^{k_1+n-1}_I$, where $S^{k_1+n-1}_I=\bigcap_{i\in I} T_i$.

\begin{pr}\label{independence}
\[V_{\bigcap_{i\in I} T_i}=\prod_{B\in \mathcal{B}(I)} V(B) = \prod_{B\in \mathcal{B}(I)} V_{\bigcap_{i \in B} T_i}.\]
\end{pr}
\begin{proof}
We will use the list of simplices $S^{k_1}_I, \ldots, S^{k_1+n-1}_I$ defined above and will show that the vertices of $S^k_I$ consist
of the first $k-k_1+2$ vertices from the list $(\ref{sequence})$ and the vertices $A^0_{k+2}, \ldots A^0_{k_1+n-1}$. 
By Proposition \ref{seq}, this statement is true for $k=k_1, \ldots, k_1+l_1$, which corresponds to the first block $B_1=\{k_1, \ldots, k_1+l_1\}$ of $I$.
Since the values $k=k_1+l_1+1, \ldots, k_2-1$ correspond to the indices that do not belong to $I$, we conclude immediately
that $S^{k_1+l_1}_I=S^{k_1+l_1+1}_I=\ldots =S^{k_2-1}_I$ and its vertices are listed correctly.

The simplex $S^{k_2}_I$ is the intersection of $S^{k_2-1}_I$ and $H_{k_2}$. Since the vertices \\ $A^0_{k_1}, A^1_{k_1}$, $\ldots, A^{l_1+1}_{k_1}, A^0_{k_1+l_1+2}, \ldots, A^0_{k_2}$
of $S^{k_2-1}_I$ belong to the convex hull of \\ $A^0_{k_1}, \ldots, A^0_{k_2}$, the only edge of $S^{k_2}_I$ which is cut by the hyperplane $\sigma_{k_2}$ is the edge
$A^0_{k_2}A^0_{k_2+1}$. This edge is replaced in $S^{k_2}$ by the edge $A^0_{k_2}A^1_{k_2}$, which confirms that the vertices of $S^{k_2}$ are listed correctly.
Taking the values of $k$ in the second block $B_2=\{k_2, \ldots,  k_2+l_2\}$, we proceed as before and always replace only one 
edge, analogously to that of the proof of Proposition $\ref{seq}$, and determine the vertices of $S^k_I$. 

Proceeding like this, in each step corresponding to $k\in I$ we replace a single edge of $S^{k-1}_I$ to obtain $S^k_I$ while each step corresponding to $k\notin I$ yields 
$S^{k-1}_I=S^k_I$.  

Having determined the vertices of the simplices $S^k_I$, the volumes $V_{S^k_I}$ are calculated easily. 
The volume $V_{S^{k_1}_I}=\frac{x_{k_1}}{1+x_{k_1}}$ by Lemma \ref{lm0}.
If $k\notin I$, then $V_{S^{k-1}_I}=V_{S^k_I}$.
If $k\in I$ and $k\neq k_1$, then $k=k_j+s$, where $0\leq s\leq l_j$ for an appropriate $j$. Then 
\[V_{S^k_I}=V_{S^{k-1}_I}t_{k_j,s+1}, \mbox{ where }
t_{k_j,s+1}=\frac{|A^{s+1}_{k_j}A^s_{k_j}|}{|A^0_{k_j+s+1}A^s_i|}.\]

Since
$V_{S^{k}_I}$ is the product of $V_{S^{k_1}_I}$ and the ratios $\frac{V_{S^{l+1}_I}}{V_{S^l_I}}$ for $l=k_1, \ldots, k-1$, it is clear that  
$V_{S^{k_1+l_1}_I}=V(B_1)$, $V_{S^{k_2+l_2}_I}=V(B_1)V(B_2), \ldots$, and
\[V_{\bigcap_{i\in I} T_i}=V_{S^{k_1+n-1}_I}=\prod_{i=1}^s V(B_i)=\prod_{B\in \mathcal{B}(I)} V_{\bigcap_{i \in B} T_i}.\]
\end{proof}

The inclusion-exclusion principle together with Propositions \ref{seq} and \ref{independence} yield the following generalization of Routh's theorem (when $n=3$) and Theorem \ref{tetra-new} (when $n=4$).

\begin{theor}\label{veta}
\[V_{\bigcap_{i=1}^n T_i}=V_{A^{n-1}_1 \ldots A^{n-1}_n}=1+\sum_{\emptyset\neq I \subsetneq \{1, \ldots, n\}} (-1)^{|I|} \prod_{B\in \mathcal{B}(I)} V(B),\]
where $|I|$ is the parity of the number of elements in $I$. 
\end{theor}

As a consequence of the above theorem we obtain the following identity.

\begin{theor}\label{identity}
\begin{equation}\label{e2}
1+\sum_{\emptyset\neq I \subsetneq \{1, \ldots, n\}} (-1)^{|I|} \prod_{B\in \mathcal{B}(I)} V(B) = 
\frac{(\prod_{i=1}^n x_i-1)^{n-1}}{\prod_{k=1}^n (1+ \sum_{b=k}^{k+n-1} \prod_{a=k}^b x_a)},
\end{equation}
where $|I|$ is the parity of the number of elements in $I$. 
\end{theor}
\begin{proof}
Assume $x_1 \cdot\ldots \cdot x_n > 1$ and consider $V_{\cap_{i=1}^n T_i}$. By Theorem \ref{veta}, this volume is given by the expression on the left-hand side of the above equation.
On the other hand, the volume of $\bigcap_{i=1}^n T_i$ can be determined using vector analysis and determinants and, by \cite{yq}, it is equal to 
\[\frac{(\prod_{i=1}^n x_i-1)^{n-1}}{\prod_{k=1}^n (1+ \sum_{b=k}^{k+n-1} \prod_{a=k}^b x_a)},\]
which is the right-hand side of the above equation.
\end{proof}

For an amusement of the reader we display the identity (\ref{e2}) for $n=5$:

\medskip

$1-\frac{x_1}{1+x_1}-\frac{x_2}{1+x_2}-\frac{x_3}{1+x_3}-\frac{x_4}{1+x_4}-\frac{x_5}{1+x_5}+\frac{x_1^2x_2}{(1+x_1)(1+x_1+x_1x_2)}
+\frac{x_2^2x_3}{(1+x_2)(1+x_2+x_2x_3)}$

$+\frac{x_3^2x_4}{(1+x_3)(1+x_3+x_3x_4)}
+\frac{x_4^2x_5}{(1+x_4)(1+x_4+x_4x_5)}
+\frac{x_5^2x_1}{(1+x_5)(1+x_5+x_5x_1)}$

$+\frac{x_1x_3}{(1+x_1)(1+x_3)}
+\frac{x_1x_4}{(1+x_1)(1+x_4)}
+\frac{x_2x_4}{(1+x_2)(1+x_4)}
+\frac{x_2x_5}{(1+x_2)(1+x_5)}
+\frac{x_3x_5}{(1+x_3)(1+x_5)}
$

$-\frac{x_1^3x_2^2x_3}{(1+x_1)(1+x_1+x_1x_2)(1+x_1+x_1x_2+x_1x_2x_3)}-\frac{x_2^3x_3^2x_4}{(1+x_2)(1+x_2+x_2x_3)(1+x_2+x_2x_3+x_2x_3x_4)}$

$-\frac{x_3^3x_4^2x_1}{(1+x_3)(1+x_3+x_3x_4)(1+x_3+x_3x_4+x_3x_4x_1)}-\frac{x_4^3x_1^2x_2}{(1+x_4)(1+x_4+x_4x_1)(1+x_4+x_4x_1+x_4x_1x_2)}$

$-\frac{x_5^3x_1^2x_2}{(1+x_5)(1+x_5+x_5x_1)(1+x_5+x_5x_1+x_5x_1x_2)}$

$-\frac{x_1^2x_2x_4}{(1+x_1)(1+x_1+x_1x_2)(1+x_4)}-\frac{x_1x_3^2x_4}{(1+x_1)(1+x_3)(1+x_3+x_3x_4)}-\frac{x_2^2x_3x_5}{(1+x_2)(1+x_2+x_2x_3)(1+x_5)}$

$-\frac{x_2x_4^2x_5}{(1+x_2)(1+x_4)(1+x_4+x_4x_5)}-\frac{x_3x_5^2x_1}{(1+x_3)(1+x_5)(1+x_5+x_5x_1)}$

$+\frac{x_1^4x_2^3x_3^2x_4}{(1+x_1)(1+x_1+x_1x_2)(1+x_1+x_1x_2+x_1x_2x_3)(1+x_1+x_1x_2+x_1x_2x_3+x_1x_2x_3x_4)}$

$+\frac{x_2^4x_3^3x_4^2x_5}{(1+x_2)(1+x_2+x_2x_3)(1+x_2+x_2x_3+x_2x_3x_4)(1+x_2+x_2x_3+x_2x_3x_4+x_2x_3x_4x_5)}$

$+\frac{x_3^4x_4^3x_5^2x_1}{(1+x_3)(1+x_3+x_3x_4)(1+x_3+x_3x_4+x_3x_4x_5)(1+x_3+x_3x_4+x_3x_4x_5+x_3x_4x_5x_1)}$

$+\frac{x_4^4x_5^3x_1^2x_2}{(1+x_4)(1+x_4+x_4x_5)(1+x_4+x_4x_5+x_4x_5x_1)(1+x_4+x_4x_5+x_4x_5x_1+x_4x_5x_1x_2)}$

$+\frac{x_5^4x_1^3x_2^2x_3}{(1+x_5)(1+x_5+x_5x_1)(1+x_5+x_5x_1+x_5x_1x_2)(1+x_5+x_5x_1+x_5x_1x_2+x_5x_1x_2x_3)}$

\bigskip

$=\frac{(x_1x_2x_3x_4x_5-1)^4}{(1+x_1+x_1x_2+x_1x_2x_3+x_1x_2x_3x_4)(1+x_2+x_2x_3+x_2x_3x_4+x_2x_3x_4x_5)
(1+x_3+x_3x_4+x_3x_4x_5+x_3x_4x_5x_1)}\times$

$\frac{1}{(1+x_4+x_4x_5+x_4x_5x_1+x_4x_5x_1x_2)
(1+x_5+x_5x_1+x_5x_1x_2+x_5x_1x_2x_3)}
$

\vskip10pt

Finally, formula (\ref{e-2}) in Theorem \ref{tetra-new}
was proven in \cite{lm} as a consequence of the identity
\[\begin{aligned}
&1-\frac{x_1}{(1+x_1)(1+x_2)(1+x_3)}-\frac{x_2}{(1+x_2)(1+x_3)(1+x_4)} -\frac{x_3}{(1+x_3)(1+x_4)(1+x_1)}\\
&-\frac{x_4}{(1+x_4)(1+x_1)(1+x_2)}
-\frac{x_1x_3}{(1+x_1)(1+x_3)}-\frac{x_2x_4}{(1+x_2)(1+x_4)}\\
&= \frac{1-x_1x_2x_3x_4}{(1+x_1)(1+x_2)(1+x_3)(1+x_4).}
\end{aligned}\]

\medskip

It would be interesting to obtain similar identities for higher dimensions. An analogous identity for $n=5$ is 

\medskip

$1-\frac{x_1}{(1+x_1) (1+x_3) (1+x_4) (1+x_5)}-\frac{x_2}{(1+x_1) (1+x_2) (1+x_4) (1+x_5)}-\frac{x_3}{(1+x_1) (1+x_2) (1+x_3) (1+x_5)}$

$- \frac{x_4}{(1+x_1) (1+x_2) (1+x_3) (1+x_4)}-\frac{x_5}{(1+x_2) (1+x_3) (1+x_4) (1+x_5)}-\frac{x_1 x_3}{(1+x_1) (1+x_3))}$

$-\frac{x_1 x_4}{(1+x_1) (1+x_4)}-\frac{x_2 x_4}{(1+x_2) (1+x_4)}-\frac{x_2 x_5}{(1+x_2) (1+x_5)}-\frac{x_3 x_5}{(1+x_3) (1+x_5)} +\frac{x_1 x_2 x_4}{(1+x_1) (1+x_2) (1+x_4)}$

$+\frac{x_1 x_3 x_4}{(1+x_1) (1+x_3) (1+x_4)}+\frac{x_1 x_3 x_5}{(1+x_1) (1+x_3) (1+x_5)}+\frac{x_2 x_3 x_5}{(1+x_2) (1+x_3) (1+x_5)}
+\frac{x_2 x_4 x_5}{(1+x_2) (1+x_4) (1+x_5)}$

$=\frac{1+x_1 x_2 x_3 x_4 x_5}{(1+x_1) (1+x_2) (1+x_3) (1+x_4) (1+x_5)}$.

\medskip

We finish by stating the formulas for the volumes of the previously considered simplices in the special case when $x_1=x_2=\dots =x_n=k$. 
In this case the volume $V=V_{\cap_{i=1}^n T_i}=\frac{|k-1|}{k^n-1}$. In particular, if $n=3$ and $k=2$, then $V=\frac17$; 
if $n=4$ and $k=2$, then $V=\frac{1}{15}$. 
The case when $n=3$ and $k=2$ is known in the literature as the area of the Feynman's triangle.

The volume of the simplex $A^1_1\ldots A^1_n$ in the special case $x_1=x_2=\dots =x_n=k$ equals 
$V_{A^1_1A^1_2A^1_3}=\frac{k^n+1}{(k+1)^n}=\frac13$ if $n=3$ and $k=2$, and equals
$V_{A^1_1A^1_2A^1_3A^1_4}=\frac{k^n-1}{(k+1)^n}=\frac{5}{27}$ if $n=4$ and $k=2$.

\medskip

{\bf Acknowledgement.} The authors are endebted to Professor Jose Alfredo Jimenez for his encouragement and help with the images presented in this article.

\end{document}